\documentclass[10pt,a4paper]{article}

\usepackage[british]{babel}
\usepackage[utf8]{inputenc}

\usepackage{tabularx}

\usepackage{amssymb,amsmath,amsthm,thmtools}
\numberwithin{equation}{section}
\usepackage{pgfplots}
\pgfplotsset{compat=1.15}
\usetikzlibrary{arrows}
\usepackage{mathtools}	
\usepackage{mathrsfs}	
\usepackage{accents}
\usepackage{multirow}
\usepackage{graphicx}
\usepackage{amscd}
\usepackage{epic, eepic}
\usepackage{url}
\usepackage{xcolor}
\usepackage{todonotes}
\usepackage{enumerate}
\usepackage{enumitem}
\usepackage{comment}
\usepackage{hhline}
\usepackage{bbm}
\usepackage[hidelinks]{hyperref}
    \addto\extrasbritish{

    }

\usepackage{tikz-cd}

\declaretheorem[style=plain,name=Theorem,numberwithin=section]{theorem}

\declaretheorem[style=plain,name=Proposition,sibling=theorem]{proposition}

\declaretheorem[style=plain,name=Lemma,sibling=theorem]{lemma}
\declaretheorem[style=plain,name=Problem,sibling=theorem]{problem}

\declaretheorem[style=plain,name=Corollary,sibling=theorem]{cor}

\declaretheorem[style=plain,name=Construction,sibling=theorem]{construction}
\declaretheorem[style=plain,name=Observation,sibling=theorem]{observation}
	
\declaretheorem[style=definition,name=Definition,sibling=theorem]{definition}

\declaretheorem[style=definition,name=Example,sibling=theorem]{example}	
\declaretheorem[style=definition,name=Question,sibling=theorem]{question}

\declaretheorem[style=definition,name=Remark,sibling=theorem]{remark}

\newcommand{\nn}{\mathbb{N}}
\newcommand{\zz}{\mathbb{Z}}

\newcommand{\rr}{\mathbb{R}}

\newcommand{\ff}{\mathbb{F}}
\newcommand{\F}{\mathbb{F}}

\newcommand{\eps}{\varepsilon}

\newcommand{\lb}{[}
\newcommand{\rb}{]}

\newcommand{\cB}{\mathcal{B}}

\newcommand{\cG}{\mathcal{G}}
\newcommand{\cN}{\mathcal{N}}
\newcommand{\cNi}{\mathcal{N}i}

\DeclareMathOperator{\Sp}{Sp}

\title{Cardinalities of the total number of independent sets}

\author{Benedek Kovács\thanks{ELTE Linear Hypergraphs Research Group, Eötvös Loránd University, Budapest, Hungary. Supported by the EKÖP-25 University Excellence Scholarship Program of the Ministry for Culture and Innovation from the source of the National Research, Development and Innovation Fund and the University Excellence Fund of Eötvös Loránd University.
		E-mail: {\tt benoke981@gmail.com}}
	\and Zolt\'an L\'or\'ant Nagy\thanks{ELTE Linear Hypergraphs  Research Group,
		E\"otv\"os Lor\'and University, Budapest, Hungary. The author is supported by the University Excellence Fund of Eötvös Loránd University.	E-mail: {\tt nagyzoli@cs.elte.hu}}
  }
\date{}

\begin{document}

\maketitle

\begin{abstract}
We study the set of numbers the total number of independent sets can admit in $n$-vertex graphs. In this paper, we prove that the cardinality $\mathcal{N}i(n)$ of this set is very close to $2^n$ in the following sense: $\mathcal{N}i(n)/2^n= O(n^{-1/5})$ while  for infinitely many $n$, we have  $\log_2(\mathcal{N}i(n)/2^n)\geq -2^{(1+o(1))\sqrt{\log_2 n}}$. This set is also precisely the set of possible values of the independence polynomial $I_G(x)$ at $x=1$ for $n$-vertex graphs $G$. As an application, we address an additive combinatorial problem on subsets of a given vector space
that avoid certain intersection patterns with respect to subspaces.
\end{abstract}

\section{Introduction}

For a graph $G=(V,E)$ on $n$ vertices, a subset $W\subseteq V$ is called \textit{independent} (or \textit{stable}) if it does not induce any edges. Denote by $I(G)$ the set of independent subsets of $G$, and let $i(G)=|I(G)|$ denote the number of independent subsets in $G$, which is also called the \textit{Fibonacci number} of the graph $G$ after Prodinger and Tichy \cite{Fibonacci}, since the sequence $(i(P_n))$ returns the Fibonacci sequence for paths $P_n$.

The investigation of  the number of independent sets in  graph families has a rich history with deep mathematical results and significant applications in various combinatorial and computational domains. Understanding these numbers and the properties of extremal graphs provides valuable insights into the structure and behaviour of graphs.
As for extremal results, we recall that  Moon and Moser proved a cornerstone theorem concerning the maximum number  of maximal independent sets in $n$-vertex graphs \cite{MM}. They also proved sharp results on the number of different sizes of maximal independent sets that an $n$-vertex graph can admit. Some related results are \cite{Erdos, NZL, Spencer, Wood}. The celebrated result of Kahn \cite{Kahn} showed that disjoint unions of complete bipartite graphs $K_{r,r}$ have the maximum number of independent sets among $r$-regular bipartite graphs on $n$ vertices, which was  extended by Zhao  to all
$r$-regular graphs \cite{Zhao}. Later, Sah, Sawhney, Stoner and Zhao described the maximum  number of independent sets in  non-regular graphs as well \cite{Sah}. 
For a relatively recent survey on counting independent sets, we refer to the work of Samotij \cite{Samotij}, see also \cite{ Mink, Cutler, DJPR}. 
Counting independent sets is a natural combinatorial problem with connections to other counting problems in graphs and hypergraphs. A very important and efficient related tool is the famous \textit{container method}, developed by  Saxton and Thomason \cite{Container2} and independently by Balogh, Morris and Samotij \cite{Container1}. It is
particularly useful when one aims to obtain upper bounds on the number of independent sets in sparse graphs or graphs with certain local properties. The core idea is to show that all such independent sets are contained in a relatively small number of larger so-called containers that have a more manageable structure, or size. 
While several applications require the understanding of \textit{the typical} \textit{or the extremal number} of independent sets in certain graph families, together with their typical or extremal structural properties, our aim is to describe the\textit{ set of all possible numbers} that can be attained as the Fibonacci number of an $n$-vertex graph.

We denote the family of (unlabelled) $n$-vertex graphs by $\cG_n$.
Since the empty set and the one-element subsets are always independent, $n+1\le i(G)\le 2^n$ clearly holds for any $G\in \cG_n$. As mentioned earlier, $\max\{ i(G): G\in \cG\}$ and $\min\{ i(G): G\in \cG\}$ have been in the centre of study for various families $\cG$ of graphs.
Our main goal is to give upper and lower bounds on the size of the whole Fibonacci spectrum 
$\{ i(G): G\in \cG_n\}$, for which we introduce

\begin{definition}
    $$\cNi(n):=|\{ i(G): G\in \cG_n\}|.$$
\end{definition}

To study the enumerative properties of independent sets, the \textit{independence polynomial} of a graph $G$, denoted by $I_G(x)$, is defined 
as the polynomial whose coefficient for $x^k$ is given by the number of independent sets of order $k$ in $G$, denoted $i_k(G)$, that is, $$I_G(x) = \sum_{k\ge 0} i_k(G)x^{k}.$$
 This polynomial was initially defined by Gutman and Harary \cite{Gutman} and is also known as the independent set polynomial or stable set polynomial. For a survey on the independence polynomial, we refer to the work of Levit and Mandrescu \cite{Levit}.
Using this terminology, our main problem is equivalent to the following.

\begin{problem}
   Determine the order of magnitude of $|\{I_G(1): G\in \cG_n\}|.$
\end{problem}
Our main result is the following.

\begin{theorem}\label{thm:main}
For all $n\in \zz^{+}$, we have
$$\cNi(n)=O(2^n n^{-0.2075}),$$
 while for infinitely many $n\in \zz^{+}$, we have 
 $$\scalebox{1.25}{$\cNi(n)\ge 2^{n-2^{(1+o(1))\sqrt{\log_2 n}}}.$}$$
\end{theorem}
Note that for every $\eps>0$,  ${2^{\sqrt{\log_2 n}}}{n^{-\eps}}\to 0$ as $n\to \infty$.

Let us outline the organization of the paper. In Section 2, we set some notation, and present some useful tools which will be applied later. In Section 3, we prove the upper bound of Theorem \ref{thm:main}. Then in Section 4 we continue by presenting general results on the so-called partial join operation and the derivation of $i(G)$ for a graph $G$ obtained via partial joining a pair of graphs $G_L$ and $G_R$, using information on $i(G')$ for induced subgraphs $G'\subseteq G_L$.
Section 5 is devoted to the presentation of the core of the main construction: the introduction of parts $G_L$, $G_R$ and a certain subfamily of their partial joins. This will show how to obtain graphs $G$ with $i(G)$ admitting values $M$, for which the binary digits of $M$ (up to an additive constant) can be arbitrarily prescribed in a certain long interval. Finally, in Section 6, we put together a construction where several copies of the previously given graph $G_R$ are combined, so that the aforementioned long interval can actually cover almost all of $[0,n]$, which in turn will imply the lower bound of Theorem \ref{thm:main}. In Section 7, we give an application of Theorem \ref{thm:main} to a general additive combinatorial problem about subsets of a given vector space that avoid certain intersection patterns with respect to subspaces \cite{codebased, KN25}. We mention that the general problem can be viewed as an extension of the cap-set problem \cite{EG}, where one wants to avoid a full (or empty) line in $\F_3^n$. Our application concerns binary vector spaces when one wants to avoid subspaces with intersection size $1$.

We conclude by raising some open problems and point out an exponential bound on the number of independence polynomials as a byproduct.

\section{Preliminaries, notation}

In this paper, a \textit{graph} always refers to a simple graph, that is, a graph without any loops or multiple edges. A graph $G=(V,E)$ has vertex set $V$ and edge set $E$. For a subset of the vertices $W\subseteq V$, the notation $G[W]$ refers to the subgraph of $G$ induced by the vertices in $W$.

A \textit{clique} is a complete graph, and the \textit{order} of the clique is its number of vertices. A clique of order $t$ is also referred to as a \textit{$t$-clique}, and is denoted $K_t$. For given disjoint vertex sets $V_1$ and $V_2$, $K[V_1,V_2]$ refers to the complete bipartite graph with classes $V_1$ and $V_2$ (or to its set of edges).

The complement of a graph $G$ is denoted by $\overline{G}$, therefore $\overline{K_t}$ denotes the empty graph on $t$ vertices.

The set of natural numbers is $\nn=\{0,1,2, \dots\}$, and the set of positive integers is $\zz^{+}=\nn\setminus \{0\}$. For a positive integer $n$, we use the notation $[n]=\{1, \dots, n\}$. For integers $m\le n$, we let $[m,n]=\{m, m+1, \dots, n\}$.

The disjoint union of sets $R_1$ and $R_2$ is denoted by $R_1\sqcup R_2$.

If $S,T\subseteq \rr$, then we use the sumset notation $S+T$ for $S+T=\{s+t: s\in S, t\in T\}$. As an abbreviation, a one-element set might be replaced by just its element, e.g., for $s\in \rr$ and $T\subseteq \rr$, $s+T=\{s+t: t\in T\}$. Similar notation is also used for subtraction or multiplication in place of addition. The sum of several sets might be abbreviated using the summation symbol, such as $\sum_{i=1}^{\ell} S_i := S_1+S_2+\ldots+S_{\ell}$.

If $B$ is an arbitrary statement, $\mathbbm{1}_B$ is the indicator function for $B$:
\begin{equation*}
\mathbbm{1}_B=
\begin{cases}
1 & \text{if $B$ is true,}\\
0 & \text{if $B$ is false.}\\
\end{cases}
\end{equation*}

In this paper, logarithms are taken base $2$. The \textit{binary entropy function} for a probability value $0\le p\le 1$, denoted $H(p)$, is defined as
$$H(p)=-p\log p-(1-p)\log (1-p).$$
We will use the following well-known bound on binomial coefficients:

\begin{lemma}\label{lem:entropy_bound}
$$\frac{1}{n+1}2^{nH\left(\frac{k}{n}\right)}\le \binom{n}{k}\le 2^{nH\left(\frac{k}{n}\right)}.$$
\end{lemma}

\noindent Given an integer $d\ge 1$, $\ff_2^d$ refers to the $d$-dimensional vector space over the binary field $\ff_2$.

\subsection{Graph operations}

Suppose we are given two graphs $G_1=(V_1,E_1)$ and $G_2=(V_2,E_2)$ on disjoint vertex sets. We define the following two standard operations on them.
\begin{itemize}
\itemsep0em
\item The \textit{disjoint union} $G_1\sqcup G_2$ is the graph $G=(V,E)$ with $V=V_1\sqcup V_2$ and $E=E_1\sqcup E_2$.
\item The \textit{full join} $G_1\triangledown G_2$ is the graph $G=(V,E)$, where $V=V_1\sqcup V_2$, and the edge set $E$ consists of $E_1\sqcup E_2$ together with all edges joining a vertex of $V_1$ to a vertex of $V_2$.
\end{itemize}

A \textit{partial join} of $G_1$ and $G_2$ is a graph obtained by adding a subset of the edges $\{v_1v_2: v_1\in V_1, v_2\in V_2\}$ to the disjoint union of $G_1$ and $G_2$, and we discuss it in more detail in Section \ref{sec:partial}.

Observe that $i(G_1\sqcup G_2)=i(G_1)i(G_2)$ and $i(G_1\triangledown G_2)=i(G_1)+i(G_2)-1$.

\subsection{Binary expansions of given support}

For $J\subseteq \nn$, let $\cB(J)$ denote the following subset of $\nn$:

$$\cB(J)=\left\{\sum_{j\in J'} 2^j: J'\subseteq J\right\}.$$

In other words, a natural number $\ell$ is in the set $\cB(J)$  if and only if 
the positions of the $1$'s in the binary representation of $\ell$ form a subset of $J$.
Note that binary digits are numbered from the right, starting with $0$.
For example, $\cB([0,n-1])=[0,2^n-1]$,  moreover, $\cB([k,\ell])=2^k[0,2^{\ell-k+1}-1]$ holds for $k\le\ell$.
We will use the 
 following simple observations without reference.
 \begin{proposition} \ - 
 \begin{itemize}
\itemsep0em
\item $|\cB(J)|=2^{|J|}$.
\item $\cB(J)=\sum_{j\in J}\{0, 2^j\}$.
\item If $J_1\cap J_2=\emptyset$ then $\cB(J_1)+\cB(J_2)=\cB(J_1\cup J_2)$.
\end{itemize}
 \end{proposition}

\subsection{Hypercube grids}\label{subsec:intro_hypercube}

For $d\in \zz^{+}$, let $I_j=[a_j,b_j]$ be an interval of integers for each $1\le j\le d$, where $a_j\le b_j$. Then the $d$-fold Cartesian product $Q:=I_1\times I_2\times \ldots\times I_d$ is called a \textit{$d$-dimensional grid}. Note that $Q$ is a subset of the integer lattice $\zz^d\subseteq \rr^d$. When $|I_1|=|I_2|=\ldots=|I_d|$, we call $Q$ a \textit{hypercube grid}, or just a \textit{hypercube} when the context is clear. The \textit{width} of the hypercube grid is $m:=|I_j|$. For given integers $j$, $k$ with $1\le j\le d$, we define $$H(j,k):=\{\mathbf{x}\in \rr^d: \mathbf{x}_j=k\},$$

\noindent which is an axis-parallel hyperplane of $\rr^d$. We will often take intersections of hypercubes with such a hyperplane: e.g. when $G$ is a hypercube and $k\in I_j$, we have $|Q\cap H(j,k)|=m^{d-1}$.

\subsection{Transversals}\label{subsec:intro_transversals}

Suppose the graph $G=(V,E)$ is a disjoint union of cliques, that is, $V=W_1\sqcup \ldots\sqcup W_{\ell}$ and for every $1\le j\le \ell$, $G[W_j]$ is isomorphic to the complete graph $K_{m_j}$ for some $m_j\in \zz^{+}$. In this setup, we will often refer to the following notions:

\begin{definition}
A \textit{full transversal} of $G$ is a set $T\subseteq V$ such that $|T\cap W_j|=1$ for every $1\le j\le \ell$.
\end{definition}
\begin{definition}
A \textit{partial transversal} of $G$ is a set $T\subseteq V$ such that $|T\cap W_j|\le 1$ for every $1\le j\le \ell$. If $\tau\subseteq [\ell]$ is the set of indices $j$ such that $|T\cap W_j|=1$, then $T$ is a \textit{partial transversal of support} $\tau$.
\end{definition}
\noindent Observe that a full transversal is a partial transversal of support $[\ell]$.

\section{Upper bound}

The main idea is the following. In order to give an upper bound on $\cNi(n)=|\{i(G): G\in \cG_n\}|$, we separate a family of non-isomorphic graphs which has cardinality less than a well-chosen function $T(n)$. Then we show that for each graph in the remaining family of graphs, the number of independent sets is also less than $T(n)$.

Recall that $\cG_n$ denotes the family of simple graphs on $n$ vertices. Given a graph $G$, let $\tau(G)$ and $\nu(G)$ denote the vertex cover number and the matching number of $G$, respectively. It is well-known that $\nu(G)\le \tau(G)\le 2\nu(G)$.

\begin{observation}\label{cut} For any function $r=r(n)$, we have $$|\{i(G): G\in \cG_n\}|\le   |\{G\in \cG_n: \tau(G)\le r\}| + |\{i(G): G\in \cG_n, \tau(G)> r\}|,$$

where in the first right-hand term, the number of graphs is counted up to isomorphism.
\end{observation}

We use the following lemma of Atmaca and Oruc \cite{Oruc}, based on Pólya's enumeration \cite{Polya}.

\begin{lemma}[\cite{Oruc}] Let $B_u(n-r,r)$ denote the set of unlabelled bipartite graphs on given bipartite classes of size $n-r$ and $r$. Then
   $$ \frac{{n-r+2^r-1\choose n-r}}{r!}\le |B_u(n-r,r)| \le 2\frac{{n-r+2^r-1\choose n-r}}{r!},$$ if $n-r > r.$
\end{lemma}

\begin{cor} \label{seged}
    If $r=\left\lceil \log \mu n\right\rceil$ for some fixed real $0<\mu<1$, then it holds for $n$ sufficiently large that
    \begin{equation*}
        \begin{split}
            |B_u(n-r,r)| &\le \frac{2}{\left\lceil\log \mu n\right\rceil!}\binom{n+2^{\left\lceil \log \mu n\right\rceil}-1}{n} \left(\frac{1}{1+\mu-\frac{1}{n}}\right)^{\log \mu n}\\
            &\le \frac{2(1+o(1))}{\sqrt{2\pi \log \mu n}}\left(\frac{e}{(1+\mu-\frac{1}{n})\log{\mu n} }\right)^{\log{\mu n}}\cdot 2^{n(1+2\mu)H\left(\frac{1}{1+2\mu}\right)}\\ 
            &\le
           (\mu n)^{
           c}\cdot 2^{n(1+2\mu)H\left(\frac{1}{1+2\mu}\right)} \text{\ \ for some absolute constant $c>0$,} 
        \end{split} 
    \end{equation*} 
   
   \noindent by using Stirling's formula and Lemma \ref{lem:entropy_bound}, after applying
\begin{equation*}
{n-r+2^r-1\choose n-r} \le {n+2^r-1\choose n}\left(\frac{1}{1+\mu-\frac{1}{n}}\right)^{\log \mu n}.\qedhere
\end{equation*}
    
 \begin{cor}\label{lem:kis_nu}
    Let $0<\mu<1$ be a real number. For $n$ sufficiently large and $r=\left\lceil \log\mu n\right\rceil$, we have
$$  |\{G\in \cG_n: \tau(G)\le r\}|\le 2^{\binom{r}{2}} |B_u(n-r,r)|\le (\mu n)^{0.5 \log{\mu n}+c}\cdot 2^{n(1+2\mu)H\left(\frac{1}{1+2\mu}\right)}$$ for some absolute constant $c>0$.
\end{cor}

\begin{proof}
The number of labelled simple graphs on $r$ vertices is $2^{\binom{r}{2}}$,  thus $2^{\binom{r}{2}} |B_u(n-r,r)|$ clearly provides an upper bound for the number of unlabelled graphs whose edges are all incident to a class of size at most $r$. Combining the upper bound $2^{\binom{r}{2}}\le(\mu n)^{0.5 \log{\mu n}+0.5}$ with Corollary \ref{seged}, we get the desired bound.
\end{proof}
\end{cor}

    To bound $|\{i(G): G\in \cG_n, \tau(G)> r\}|$, observe that any $G\in \cG_n$ with vertex cover number $\tau(G)>r$ has at least $\ell:=\left\lceil\frac{r+1}{2}\right\rceil$ independent edges. The two endpoints of any of these edges cannot both appear in an independent set of $G$, therefore $i(G)\le 3^{\ell}\cdot 2^{n-2\ell}$ for any such $G$. Hence we in turn have 
        
    \begin{lemma} \label{lem:nagy_nu}
       $$|\{i(G): G\in \cG_n, \tau(G)> r\}|\le 3^{\ell}\cdot 2^{n-2\ell}\le 2^n\left(\frac34\right)^{r/2}.$$
    \end{lemma}

\begin{proof}[Proof of Theorem \ref{thm:main}, upper bound]
    Let $\mu_0$ be the solution of the equality $(1+2\mu_0)H\left(\frac{1}{1+2\mu_0}\right)=1.$ This gives $\mu_0\approx 0.146908.$ Let us fix $\mu=0.1469<\mu_0$.  
    Then, by taking $r(n)=\left\lceil \log \mu n\right\rceil$ in Proposition \ref{cut}, for $n$ large enough we get  $$ |\{i(G): G\in \cG_n\}|\le 2^{0.99997n} + 2^n\cdot 2^{0.5\log(0.75)\log (0.1469n)}=O(2^n n^{-0.2075})$$ in view of Corollary \ref{lem:kis_nu} and Lemma  \ref{lem:nagy_nu}.
\end{proof}

\section{Partial joins and the summation trick}\label{sec:partial}

Our graph constructions will be reliant on the concept of partial joins.

\begin{definition}[Partial join of two graphs]
Let $G_L$ and $G_R$ be two graphs on disjoint vertex sets $V_L$ and $V_R$, respectively. A \textit{partial join} of $G_L$ and $G_R$ is any graph $G=(V,E)$ satisfying the following properties:
\begin{itemize}
\itemsep0em
\item $V=V_L\sqcup V_R$,
\item $G[V_L]=G_L$ and $G[V_R]=G_R$.
\end{itemize}
In this setup, $G_L$ and $G_R$ will often be called the \textit{left-hand} and \textit{right-hand classes}.
\end{definition}

In other words, a partial join is a graph $G$ on the combined vertex set $V=V_L\sqcup V_R$ such that all edges and non-edges are preserved within the classes $V_L$ and $V_R$, and for any pair $v_L\in V_L$ and $v_R\in V_R$, we can freely choose whether to include the edge $v_Lv_R$ in $G$ or not. If we include all such edges, the resulting graph $G$ is the full join $G_L\triangledown G_R$ of the two graphs. On the other hand, if we do not include any of these edges, we get the disjoint union $G_L\sqcup G_R$ of the two graphs.

In the main construction of this paper, $G_L$ will be a large empty graph, while $G_R$ will be a disjoint union of several cliques. Different values $i(G)$ can then be obtained by appropriately setting the edges between $V_L$ and $V_R$, providing the variable part of the construction.

The following trick allows us to count independent subsets in a partial join:

\begin{proposition}[Summation trick]\label{prop:summationtrick}
Let $G_L$ and $G_R$ be arbitrary graphs with disjoint vertex sets $V_L$ and $V_R$, and let $G$ be a partial join of $G_L$ and $G_R$. For every $v_R\in V_R$, denote by $N_L(v_R)$ the set of vertices in $V_L$ adjacent to $v_R$. Then
$$i(G)=\sum_{S_R\in I(G_R)} i\left(G_L\left[V_L\setminus \bigcup_{v_R\in S_R} N_L(v_R)\right]\right).$$
\end{proposition}
\begin{proof}
We count independent subsets $S$ of $G$ by considering each possible value of $S_R:=S\cap V_R$ as a separate case. Clearly, $S_R$ has to be an independent set in $G_R$, and if $S_R$ is given, then $S$ cannot contain any vertex from $V_L$ that is a neighbour of any vertex in $S_R$, therefore $S_L:=S\cap V_L$ has to be a subset of $V_L\setminus \bigcup\limits_{v_R\in S_R} N_L(v_R)$. As $S_L$ also has to be independent, we indeed get that there are $i(G_L[V_L\setminus\bigcup\limits_{v_R\in S_R} N_L(v_R)])$ independent subsets $S$ for this fixed $S_R$. Summing over all $S_R\in I(G_R)$ proves the result.
\end{proof}

\begin{example}\label{ex:oneclique}
A simple but illustrative case for the usage of Proposition \ref{prop:summationtrick} is the case when $V_R$ is a clique, say $V_R=\{v_1, \ldots, v_m\}$. In this case, we get
$$i(G)=i(G_L)+i(G_L[V_L\setminus N_1])+\ldots+i(G_L[V_L\setminus N_m]),$$

\noindent where the subsets $N_j:=N_L(v_j)$ for each $1\le j\le m$ are arbitrarily chosen subsets of $V_L$.
\end{example}

\begin{example}\label{ex:multiplecliques}
If $G_R$ is a disjoint union of cliques, this provides a more subtle way of summing Fibonacci numbers of induced subgraphs of $G_L$. Say $G_R$ is a disjoint union of $\ell$ cliques of sizes $m_1$, \ldots, $m_{\ell}$, where the $j$-th clique is denoted $W_j=\{w_{j,1}, \ldots, w_{j,m_j}\}$. Again, let $N_{j,k}:=N_L(w_{j,k})$ be arbitrary subsets of $V_L$ for each $1\le j\le \ell$ and $1\le k\le m_j$.

Independent subsets of $G_R$ are then precisely those sets $T_R\subseteq V_R$ that contain at most one element from each clique $W_j$, including the case $T_R=\emptyset$ as well. That is, $I(G_R)$ consists of precisely the partial transversals of $G_R$. Denoting by $\tau$ the support of such a partial transversal $T_R$, we get
\begin{align}
i(G) &= \sum_{T_R} \left(i\left(G_L\left[V_L\setminus \bigcup_{v_R\in T_R} N_L(v_R)\right]\right): \text{$T_R$ is a partial transversal of $G_R$}\right)\notag\\
~ &= \sum_{\tau\subseteq [\ell]}  \sum_{\substack{f:~\tau\to \zz^{+} \\ f(j)\in [m_j] ~\forall j}} i\left(G_L\left[V_L\setminus \bigcup_{t\in \tau} N_{t,f(t)}\right]\right).\label{eq:disjointtrick}
\end{align}

\noindent Figure \ref{fig:partialjoin} shows an example of this construction with $G_L=\overline{K_{15}}$ and $G_R=4K_3$.

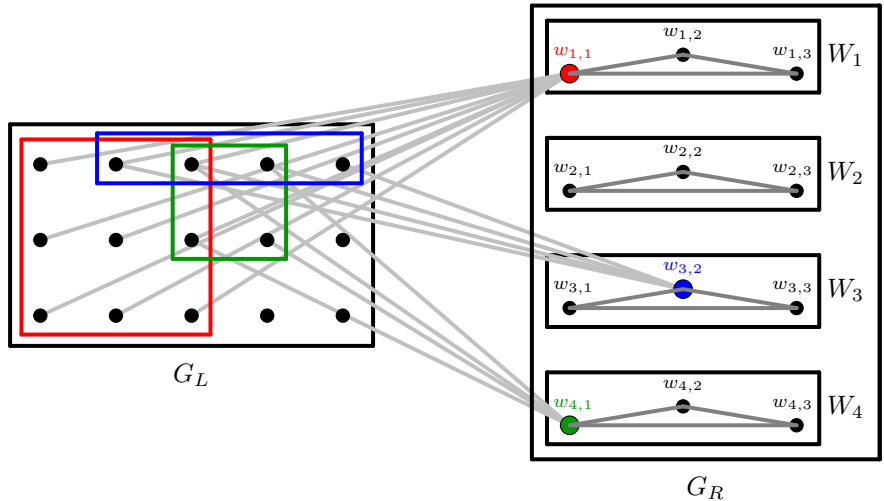
\begin{figure}[ht]
 \centering
\definecolor{cqcqcq}{rgb}{0.7529411764705882,0.7529411764705882,0.7529411764705882}
\definecolor{yqyqyq}{rgb}{0.5019607843137255,0.5019607843137255,0.5019607843137255}
\definecolor{qqqqff}{rgb}{0.,0.,1.}
\definecolor{qqzzqq}{rgb}{0.,0.6,0.}
\definecolor{ffqqqq}{rgb}{1.,0.,0.}
\begin{tikzpicture}[line cap=round,line join=round,>=triangle 45,x=1.0cm,y=1.0cm]
\clip(-0.7,-2.7) rectangle (11.5,4.5);
\draw [line width=1.5pt] (-0.4,2.53)-- (4.4,2.53);
\draw [line width=1.5pt] (4.4,2.53)-- (4.4,-0.4);
\draw [line width=1.5pt] (4.4,-0.4)-- (-0.4,-0.4);
\draw [line width=1.5pt] (-0.4,-0.4)-- (-0.4,2.53);
\draw (2,-0.8) node[anchor=mid] {$G_L$};
\draw [line width=1.5pt] (6.7,2.35)-- (6.7,1.4);
\draw [line width=1.5pt] (6.7,1.4)-- (10.3,1.4);
\draw [line width=1.5pt] (10.3,1.4)-- (10.3,2.35);
\draw [line width=1.5pt] (10.3,2.35)-- (6.7,2.35);
\draw [line width=1.5pt] (6.7,2.35)-- (10.3,2.35);
\draw [line width=1.5pt] (6.7,0.8)-- (6.7,-0.15);
\draw [line width=1.5pt] (6.7,-0.15)-- (10.3,-0.15);
\draw [line width=1.5pt] (10.3,-0.15)-- (10.3,0.8);
\draw [line width=1.5pt] (10.3,0.8)-- (6.7,0.8);
\draw [line width=1.5pt] (6.7,-0.75)-- (6.7,-1.7);
\draw [line width=1.5pt] (6.7,-1.7)-- (10.3,-1.7);
\draw [line width=1.5pt] (10.3,-1.7)-- (10.3,-0.75);
\draw [line width=1.5pt] (10.3,-0.75)-- (6.7,-0.75);
\draw [line width=1.5pt] (6.7,3.9)-- (6.7,2.95);
\draw [line width=1.5pt] (6.7,2.95)-- (10.3,2.95);
\draw [line width=1.5pt] (10.3,2.95)-- (10.3,3.9);
\draw [line width=1.5pt] (10.3,3.9)-- (6.7,3.9);
\draw [line width=1.5pt,color=cqcqcq] (7.,3.2)-- (0.,2.);
\draw [line width=1.5pt,color=cqcqcq] (1.,2.)-- (7.,3.2);
\draw [line width=1.5pt,color=cqcqcq] (7.,3.2)-- (2.,2.);
\draw [line width=1.5pt,color=cqcqcq] (0.,1.)-- (7.,3.2);
\draw [line width=1.5pt,color=cqcqcq] (7.,3.2)-- (1.,1.);
\draw [line width=1.5pt,color=cqcqcq] (2.,1.)-- (7.,3.2);
\draw [line width=1.5pt,color=cqcqcq] (7.,3.2)-- (2.,0.);
\draw [line width=1.5pt,color=cqcqcq] (1.,0.)-- (7.,3.2);
\draw [line width=1.5pt,color=cqcqcq] (7.,3.2)-- (0.,0.);
\draw [line width=1.5pt,color=cqcqcq] (8.5,0.35)-- (1.,2.);
\draw [line width=1.5pt,color=cqcqcq] (8.5,0.35)-- (2.,2.);
\draw [line width=1.5pt,color=cqcqcq] (8.5,0.35)-- (3.,2.);
\draw [line width=1.5pt,color=cqcqcq] (8.5,0.35)-- (4.,2.);
\draw [line width=1.5pt,color=cqcqcq] (7.,-1.45)-- (2.,1.);
\draw [line width=1.5pt,color=cqcqcq] (7.,-1.45)-- (3.,1.);
\draw [line width=1.5pt,color=cqcqcq] (7.,-1.45)-- (2.,2.);
\draw [line width=1.5pt,color=cqcqcq] (7.,-1.45)-- (3.,2.);
\draw [line width=1.5pt,color=ffqqqq] (-0.25,2.33)-- (2.25,2.33);
\draw [line width=1.5pt,color=ffqqqq] (2.25,2.33)-- (2.25,-0.25);
\draw [line width=1.5pt,color=ffqqqq] (2.25,-0.25)-- (-0.25,-0.25);
\draw [line width=1.5pt,color=ffqqqq] (-0.25,-0.25)-- (-0.25,2.33);
\draw [line width=1.5pt,color=qqzzqq] (1.75,2.25)-- (3.25,2.25);
\draw [line width=1.5pt,color=qqzzqq] (3.25,2.25)-- (3.25,0.75);
\draw [line width=1.5pt,color=qqzzqq] (3.25,0.75)-- (1.75,0.75);
\draw [line width=1.5pt,color=qqzzqq] (1.75,0.75)-- (1.75,2.25);
\draw [line width=1.5pt,color=qqqqff] (0.75,2.41)-- (4.25,2.41);
\draw [line width=1.5pt,color=qqqqff] (4.25,2.41)-- (4.25,1.75);
\draw [line width=1.5pt,color=qqqqff] (4.25,1.75)-- (0.75,1.75);
\draw [line width=1.5pt,color=qqqqff] (0.75,1.75)-- (0.75,2.41);
\draw [line width=1.5pt] (6.5,4.1)-- (11.1,4.1);
\draw [line width=1.5pt] (11.1,4.1)-- (11.1,-1.9);
\draw [line width=1.5pt] (11.1,-1.9)-- (6.5,-1.9);
\draw [line width=1.5pt] (6.5,-1.9)-- (6.5,4.1);
\draw [line width=1.5pt] (11.1,4.1)-- (11.1,-1.9);
\draw (10.65,3.425) node[anchor=mid] {$W_1$};
\draw (10.65,1.875) node[anchor=mid] {$W_2$};
\draw (10.65,0.325) node[anchor=mid] {$W_3$};
\draw (10.65,-1.225) node[anchor=mid] {$W_4$};
\draw (8.8,-2.3) node[anchor=mid] {$G_R$};
\begin{scriptsize}
\draw [fill=black] (0.,0.) circle (2.5pt);
\draw [fill=black] (1.,0.) circle (2.5pt);
\draw [fill=black] (2.,0.) circle (2.5pt);
\draw [fill=black] (3.,0.) circle (2.5pt);
\draw [fill=black] (4.,0.) circle (2.5pt);
\draw [fill=black] (0.,1.) circle (2.5pt);
\draw [fill=black] (1.,1.) circle (2.5pt);
\draw [fill=black] (2.,1.) circle (2.5pt);
\draw [fill=black] (3.,1.) circle (2.5pt);
\draw [fill=black] (4.,1.) circle (2.5pt);
\draw [fill=black] (0.,2.) circle (2.5pt);
\draw [fill=black] (1.,2.) circle (2.5pt);
\draw [fill=black] (2.,2.) circle (2.5pt);
\draw [fill=black] (3.,2.) circle (2.5pt);
\draw [fill=black] (4.,2.) circle (2.5pt);
\draw [fill=black] (7.,1.65) circle (2.5pt);
\draw[color=black] (7.05,1.95) node[anchor=mid] {$w_{2,1}$};
\draw [fill=black] (8.5,1.9) circle (2.5pt);
\draw[color=black] (8.5,2.2) node[anchor=mid] {$w_{2,2}$};
\draw [fill=black] (10.,1.65) circle (2.5pt);
\draw[color=black] (9.95,1.95) node[anchor=mid] {$w_{2,3}$};
\draw [fill=black] (7.,0.1) circle (2.5pt);
\draw[color=black] (7.05,0.4) node[anchor=mid] {$w_{3,1}$};
\draw [fill=qqqqff] (8.5,0.35) circle (3.5pt);
\draw[color=qqqqff] (8.5,0.65) node[anchor=mid] {$w_{3,2}$};
\draw [fill=black] (10.,0.1) circle (2.5pt);
\draw[color=black] (9.95,0.4) node[anchor=mid] {$w_{3,3}$};
\draw [fill=qqzzqq] (7.,-1.45) circle (3.5pt);
\draw[color=qqzzqq] (7.05,-1.15) node[anchor=mid] {$w_{4,1}$};
\draw [fill=black] (8.5,-1.2) circle (2.5pt);
\draw[color=black] (8.5,-0.9) node[anchor=mid] {$w_{4,2}$};
\draw [fill=black] (10.,-1.45) circle (2.5pt);
\draw[color=black] (9.95,-1.15) node[anchor=mid] {$w_{4,3}$};
\draw [fill=ffqqqq] (7.,3.2) circle (3.5pt);
\draw[color=ffqqqq] (7.05,3.5) node[anchor=mid] {$w_{1,1}$};
\draw [fill=black] (8.5,3.45) circle (2.5pt);
\draw[color=black] (8.5,3.75) node[anchor=mid] {$w_{1,2}$};
\draw [fill=black] (10.,3.2) circle (2.5pt);
\draw[color=black] (9.95,3.5) node[anchor=mid] {$w_{1,3}$};
\end{scriptsize}
\draw [line width=1.5pt,color=yqyqyq] (7.,1.65)-- (8.5,1.9);
\draw [line width=1.5pt,color=yqyqyq] (8.5,1.9)-- (10.,1.65);
\draw [line width=1.5pt,color=yqyqyq] (10.,1.65)-- (7.,1.65);
\draw [line width=1.5pt,color=yqyqyq] (7.,3.2)-- (8.5,3.45);
\draw [line width=1.5pt,color=yqyqyq] (8.5,3.45)-- (10.,3.2);
\draw [line width=1.5pt,color=yqyqyq] (10.,3.2)-- (7.,3.2);
\draw [line width=1.5pt,color=yqyqyq] (7.,0.1)-- (8.5,0.35);
\draw [line width=1.5pt,color=yqyqyq] (8.5,0.35)-- (10.,0.1);
\draw [line width=1.5pt,color=yqyqyq] (10.,0.1)-- (7.,0.1);
\draw [line width=1.5pt,color=yqyqyq] (7.,-1.45)-- (8.5,-1.2);
\draw [line width=1.5pt,color=yqyqyq] (8.5,-1.2)-- (10.,-1.45);
\draw [line width=1.5pt,color=yqyqyq] (10.,-1.45)-- (7.,-1.45);
\end{tikzpicture}
\caption{An example for a partial join of $G_L=\overline{K_{15}}$ and $G_R=4K_3$. For the red, blue and green vertices of $G_R$, their neighbour sets in $G_L$ are denoted with identically coloured boxes, whereas the rest of the vertices of $G_R$ have no neighbours in $G_L$. The partial transversal $T_R=\{w_{1,1}, w_{3,2}, w_{4,1}\}$ consisting of the three coloured vertices has $|V_L\setminus \cup_{v_R\in T_R} N_L(v_R)|=3$, therefore it contributes $i(\overline{K_3})=8$ to the total sum $i(G)$ shown in \eqref{eq:disjointtrick}. In contrast, the full transversal $T'_R=\{w_{1,1}, w_{2,1}, w_{3,1}, w_{4,1}\}$ would contribute $2^4=16$.}
\label{fig:partialjoin}
\end{figure}
\end{example}

\subsection{Partial join spectra}

\begin{definition}
Let $G_L$ and $G_R$ be two graphs with disjoint vertex sets $V_L$ and $V_R$. Then the \textit{partial join spectrum} of $G_L$ and $G_R$ is defined as
$$\Sp(G_L,G_R):=\{i(G): G\text{ is a partial join of }G_L\text{ and }G_R\}.$$
\end{definition}

The following identities will be useful when the left-hand class is augmented with additional isolated vertices.

\begin{lemma}\label{lem:spectrumscaling}
Let $G_L$ and $G_R$ be two graphs on disjoint vertex sets, and $t\in \nn$. Then
\begin{enumerate}[label=(\alph*)]
\item $\Sp(G_L\sqcup \overline{K_t}, G_R)\supseteq 2^t\Sp(G_L, G_R)$,
\item $\Sp(G_L\sqcup \overline{K_t}, G_R) \supseteq \Sp(G_L, G_R) + (2^t-1)i(G_L)$.
\end{enumerate}
\end{lemma}
\begin{proof}
\begin{enumerate}[label=(\alph*)]
\item Let $G$ be any partial join of $G_L$ and $G_R$, and let $G'$ be the partial join of $G_L\sqcup \overline{K_t}$ and $G_R$ obtained from $G$ by adding $t$ isolated vertices to the left-hand class that are not joined to any right-hand vertices either. Then $i(G')=i(G\sqcup \overline{K_t})=2^t\cdot i(G)$.
\item Let $G$ be any partial join of $G_L$ and $G_R$, and let $G'$ be the partial join of $G_L\sqcup \overline{K_t}$ and $G_R$ obtained from $G$ by adding $t$ isolated vertices to the left-hand class and joining all of them to all right-hand vertices. Then the independent sets of $G'$ are just the independent sets of $G$, with the added case that when we do not choose any right-hand vertices, we can pick any independent set of $G_L$ together with a nonempty subset of $\overline{K_t}$, giving $(2^t-1)i(G_L)$ further options.
\qedhere
\end{enumerate}
\end{proof}

The following lemma combines several partial joins into one, and provides the framework used for putting together our full construction.

\begin{lemma}\label{lem:join_with_multiple}
For $\ell\ge 1$, suppose that $G_L$, $G^{(1)}_R$, $G^{(2)}_R$, \dots, $G^{(\ell)}_R$ are vertex-disjoint graphs. Let 
$$G_R=G^{(1)}_R\triangledown G^{(2)}_R \triangledown \dots\triangledown G^{(\ell)}_R$$ be the full join of the $\ell$ graphs $G^{(1)}_R$, $G^{(2)}_R$, \dots, $G^{(\ell)}_R$. Then
$$\Sp(G_L,G_R)=\Sp(G_L,G^{(1)}_R) + \ldots + \Sp(G_L,G^{(\ell)}_R) - (\ell-1)\cdot i(G_L).$$
\end{lemma}
\begin{proof}
Let $V_L$, $V_R$, and $V^{(j)}_R$ (for $1\le j\le \ell$) denote the vertex sets of graphs $G_L$, $G_R$ and $G^{(j)}_R$, respectively.

Consider any partial join $G$ of $G_L$ and $G_R$, and for each $1\le j\le \ell$, let $G^{(j)}$ be the subgraph of $G$ induced by $V_L\sqcup V^{(j)}_R$. Then $G^{(j)}$ is an arbitrary partial join of $G_L$ and $G^{(j)}_R$.

For an independent subset $S$ of $G$, either $S\cap V_R=\emptyset$ or not. In the former case, $S$ is an arbitrary independent subset of $G_L$. In the latter case, since $G_R$ is a full join, there is exactly one $j$ such that $S\cap V_R^{(j)}\ne \emptyset$, and $S$ is an arbitrary independent subset of $G^{(j)}$ not fully contained in $V_L$. Hence the overall number of possibilities for $S$ is
$$i(G)=i(G_L) + \sum_{j=1}^{\ell} \left(i(G^{(j)}) - i(G_L)\right)=\sum_{j=1}^{\ell} i(G^{(j)}) - (\ell-1)\cdot i(G_L).$$

Since the partial joins $G^{(j)}$ are arbitrary, we get the desired result for the partial join spectra.
\end{proof}

\section{Construction: main part}\label{sec:constr_main}
The goal of this section is to prove the following proposition, which will provide the main part of our construction. We give two graphs such that in their set of partial joins $G$, the value $i(G)-c$ (for some constant $c$) has a long interval of binary digits that can be arbitrarily set by varying the edges between the two classes.

\begin{proposition}\label{prop:long_interval}
Fix an absolute constant $\eps>0$, and take $d,m\in \zz^{+}$ such that $d\ge d_0(\eps)$ is large enough and $m>2^d$. Then for the two graphs $G_L=\overline{K_{m^{d+1}}}$ and $G_R=(d+1)K_m$, we have
$$\Sp(G_L,G_R)\supseteq \cB\left(\left[\left\lceil(1+\eps)2^{d-1}m^d\right\rceil,~ \left\lfloor m^{d+1}-(1+\eps)2^{d-1}m^d\right\rfloor\right]\right)+c,$$

where $c\in \nn$ is a constant depending on $d,m,\eps$.
\end{proposition}

In this section, we consider the parameters $\eps>0$, $d\ge d_0$ and $m>2^d$ to be fixed, where $d_0\ge 3$ is to be determined later. We will consider partial joins $G$ of $G_L=\overline{K_{m^{d+1}}}$ and $G_R=(d+1)K_m$, where some adjacencies between the two classes will be fixed and some will be variable, providing various possibilities for $i(G)$.

The left-hand class is a large empty graph, which we will present as a disjoint union of various sets. We will give a labelling of the vertices in these sets, which will be crucial to describing the adjacencies between the two classes.

The main mechanism of the labelling of $G_L$ will be the so-called \textit{extended hypercube}, whose properties will be presented in Subsection \ref{subsec:extended_hypercube}. Then in Subsection \ref{subsec:vertex_labellings}, we present the vertices and edges of $G$. In Subsections \ref{subsec:computation} and \ref{subsec:fixingtkchoices}, we compute $i(G)$, focusing on its variable term. It turns out that this term can be set to any value in $\cB(J)$ for a long interval $J$, when certain parameters of the construction are appropriately set.

\subsection{The extended hypercube}\label{subsec:extended_hypercube}

According to the notation in Subsection \ref{subsec:intro_hypercube}, we consider the hypercubes $Q:=[m]^d$ and $\widehat{Q}:=[2m]^d$, which satisfy $Q\subseteq \widehat{Q}$. We call $Q$ and $\widehat{Q}$ the \textit{base hypercube} and the \textit{extended hypercube}, respectively.

Two points $\mathbf{y}$ and $\mathbf{y}'$ in $\widehat{Q}$ will be called \textit{congruent} if for every $1\le j\le d$, we have $\mathbf{y}_j\equiv \mathbf{y}'_j\pmod{m}$. In this way, each point $\mathbf{x}\in Q$ has $2^d$ points congruent to it in $\widehat{Q}$, including itself: these points can be written $\mathbf{x}+m\mathbf{v}$ for some $\mathbf{v}\in \{0,1\}^d$.

The extended hypercube clearly satisfies the following intersection property with axis-parallel hyperplanes. Recall that $H(j,k)=\{\mathbf{x}\in \rr^d: \mathbf{x}_j=k\}$.
\begin{observation}\label{obs:hypercube_int_property}
Let $0\le \ell\le d$ be an integer. Given $J=\{j_1, j_2, \ldots, j_{\ell}\}\subseteq [d]$ and integers $k_1, k_2, \ldots, k_{\ell}\in [2m]$, the following holds:
\begin{equation*}
|\widehat{Q}\cap H(j_1, k_1)\cap \ldots\cap H(j_{\ell},k_{\ell})|=(2m)^{d-\ell}.
\end{equation*}
\end{observation}

We now introduce a way to extend an arbitrary subset of $Q$ to $\widehat{Q}$ to obtain a set satisfying a similar intersection property with hyperplanes. For $\mathbf{v}\in \{0,1\}^d$, call $\mathbf{v}$ \textit{even}, or \textit{odd}, if the sum of its coordinates is even, or odd, respectively. If the even elements of $\{0,1\}^d$ are coloured dark and the odd elements light, this resembles a checkerboard pattern.

\begin{definition}[Checkered extension]\label{def:checkering}
Given a subset $S\subseteq Q$, define the \textit{checkered extension} of $S$ to be the following subset $S_c\subseteq \widehat{Q}$:
\begin{equation*}
\begin{split}
S_c &=\{\mathbf{y}=\mathbf{x}+m\mathbf{v}: (\mathbf{x}\in S, ~\mathbf{v}\in \{0,1\}^d, ~\mathbf{v}\text{ is even})\} \\
&\sqcup \{\mathbf{y}=\mathbf{x}+m\mathbf{v}: (\mathbf{x}\in Q\setminus S, ~\mathbf{v}\in \{0,1\}^d, ~\mathbf{v}\text{ is odd})\}.
\end{split}
\end{equation*}
\end{definition}

Figure~\ref{fig:checkerboard} shows an example for the checkered extension of a set in the case $d=2$.

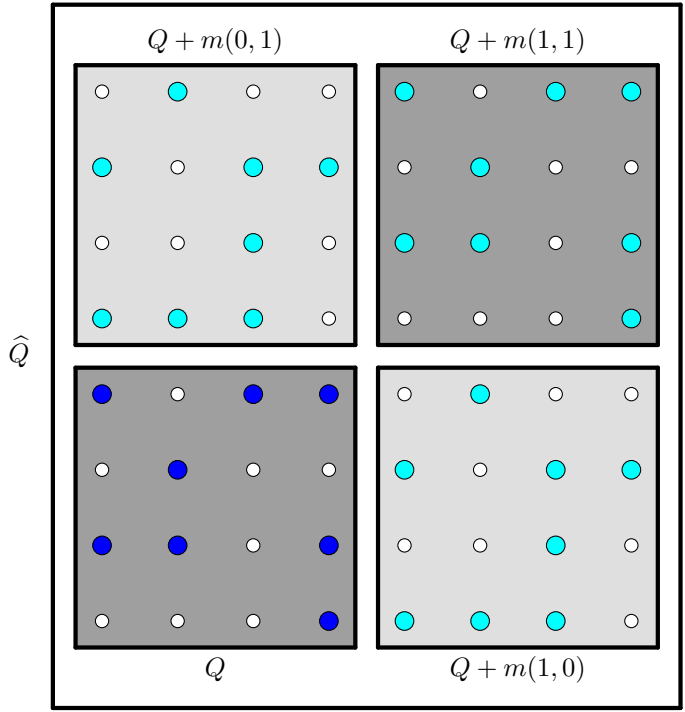
\begin{figure}
 \centering
\definecolor{cqcqcq}{rgb}{0.7529411764705882,0.7529411764705882,0.7529411764705882}
\definecolor{uququq}{rgb}{0.25098039215686274,0.25098039215686274,0.25098039215686274}
\definecolor{qqqqff}{rgb}{0.,0.,1.}
\definecolor{qqffff}{rgb}{0.,1.,1.}
\definecolor{ffffff}{rgb}{1.,1.,1.}
\begin{tikzpicture}[line cap=round,line join=round,>=triangle 45,x=1.0cm,y=1.0cm]
\clip(-1.4,-1.4) rectangle (8.4,8.4);
\fill[line width=0.pt,color=cqcqcq,fill=cqcqcq,fill opacity=0.5] (-0.35,7.35) -- (3.35,7.35) -- (3.35,3.65) -- (-0.35,3.65) -- cycle;
\fill[line width=0.pt,color=cqcqcq,fill=cqcqcq,fill opacity=0.5] (3.65,3.35) -- (7.35,3.35) -- (7.35,-0.35) -- (3.65,-0.35) -- cycle;
\fill[line width=0.pt,color=uququq,fill=uququq,fill opacity=0.5] (3.65,7.35) -- (7.35,7.35) -- (7.35,3.65) -- (3.65,3.65) -- cycle;
\fill[line width=0.pt,color=uququq,fill=uququq,fill opacity=0.5] (-0.35,3.35) -- (3.35,3.35) -- (3.35,-0.35) -- (-0.35,-0.35) -- cycle;
\draw [line width=1.6pt] (-0.35,7.35)-- (3.35,7.35);
\draw [line width=1.6pt] (3.35,7.35)-- (3.35,3.65);
\draw [line width=1.6pt] (3.35,3.65)-- (-0.35,3.65);
\draw [line width=1.6pt] (-0.35,3.65)-- (-0.35,7.35);
\draw [line width=1.6pt] (3.65,7.35)-- (3.65,3.65);
\draw [line width=1.6pt] (3.65,3.65)-- (7.35,3.65);
\draw [line width=1.6pt] (7.35,3.65)-- (7.35,7.35);
\draw [line width=1.6pt] (7.35,7.35)-- (3.65,7.35);
\draw [line width=1.6pt] (-0.35,3.35)-- (3.35,3.35);
\draw [line width=1.6pt] (3.35,3.35)-- (3.35,-0.35);
\draw [line width=1.6pt] (3.35,-0.35)-- (-0.35,-0.35);
\draw [line width=1.6pt] (-0.35,3.35)-- (-0.35,-0.35);
\draw [line width=1.6pt] (3.65,3.35)-- (3.65,-0.35);
\draw [line width=1.6pt] (3.65,-0.35)-- (7.35,-0.35);
\draw [line width=1.6pt] (7.35,-0.35)-- (7.35,3.35);
\draw [line width=1.6pt] (7.35,3.35)-- (3.65,3.35);
\draw [line width=1.6pt] (-0.65,8.15)-- (7.65,8.15);
\draw [line width=1.6pt] (7.65,8.15)-- (7.65,-1.15);
\draw [line width=1.6pt] (7.65,-1.15)-- (-0.65,-1.15);
\draw [line width=1.6pt] (-0.65,-1.15)-- (-0.65,8.15);
\begin{scriptsize}
\draw [fill=ffffff] (0.,0.) circle (2.5pt);
\draw [fill=qqqqff] (0.,1.) circle (3.5pt);
\draw [fill=ffffff] (0.,2.) circle (2.5pt);
\draw [fill=qqqqff] (0.,3.) circle (3.5pt);
\draw [fill=qqffff] (0.,4.) circle (3.5pt);
\draw [fill=ffffff] (0.,5.) circle (2.5pt);
\draw [fill=qqffff] (0.,6.) circle (3.5pt);
\draw [fill=ffffff] (0.,7.) circle (2.5pt);
\draw [fill=ffffff] (1.,0.) circle (2.5pt);
\draw [fill=qqqqff] (1.,1.) circle (3.5pt);
\draw [fill=qqqqff] (1.,2.) circle (3.5pt);
\draw [fill=ffffff] (1.,3.) circle (2.5pt);
\draw [fill=qqffff] (1.,4.) circle (3.5pt);
\draw [fill=ffffff] (1.,5.) circle (2.5pt);
\draw [fill=ffffff] (1.,6.) circle (2.5pt);
\draw [fill=qqffff] (1.,7.) circle (3.5pt);
\draw [fill=ffffff] (2.,0.) circle (2.5pt);
\draw [fill=ffffff] (2.,1.) circle (2.5pt);
\draw [fill=ffffff] (2.,2.) circle (2.5pt);
\draw [fill=qqqqff] (2.,3.) circle (3.5pt);
\draw [fill=qqffff] (2.,4.) circle (3.5pt);
\draw [fill=qqffff] (2.,5.) circle (3.5pt);
\draw [fill=qqffff] (2.,6.) circle (3.5pt);
\draw [fill=ffffff] (2.,7.) circle (2.5pt);
\draw [fill=qqqqff] (3.,0.) circle (3.5pt);
\draw [fill=qqqqff] (3.,1.) circle (3.5pt);
\draw [fill=ffffff] (3.,2.) circle (2.5pt);
\draw [fill=qqqqff] (3.,3.) circle (3.5pt);
\draw [fill=ffffff] (3.,4.) circle (2.5pt);
\draw [fill=ffffff] (3.,5.) circle (2.5pt);
\draw [fill=qqffff] (3.,6.) circle (3.5pt);
\draw [fill=ffffff] (3.,7.) circle (2.5pt);
\draw [fill=qqffff] (4.,0.) circle (3.5pt);
\draw [fill=ffffff] (4.,1.) circle (2.5pt);
\draw [fill=qqffff] (4.,2.) circle (3.5pt);
\draw [fill=ffffff] (4.,3.) circle (2.5pt);
\draw [fill=ffffff] (4.,4.) circle (2.5pt);
\draw [fill=qqffff] (4.,5.) circle (3.5pt);
\draw [fill=ffffff] (4.,6.) circle (2.5pt);
\draw [fill=qqffff] (4.,7.) circle (3.5pt);
\draw [fill=qqffff] (5.,0.) circle (3.5pt);
\draw [fill=ffffff] (5.,1.) circle (2.5pt);
\draw [fill=ffffff] (5.,2.) circle (2.5pt);
\draw [fill=qqffff] (5.,3.) circle (3.5pt);
\draw [fill=ffffff] (5.,4.) circle (2.5pt);
\draw [fill=qqffff] (5.,5.) circle (3.5pt);
\draw [fill=qqffff] (5.,6.) circle (3.5pt);
\draw [fill=ffffff] (5.,7.) circle (2.5pt);
\draw [fill=qqffff] (6.,0.) circle (3.5pt);
\draw [fill=qqffff] (6.,1.) circle (3.5pt);
\draw [fill=qqffff] (6.,2.) circle (3.5pt);
\draw [fill=ffffff] (6.,3.) circle (2.5pt);
\draw [fill=ffffff] (6.,4.) circle (2.5pt);
\draw [fill=ffffff] (6.,5.) circle (2.5pt);
\draw [fill=ffffff] (6.,6.) circle (2.5pt);
\draw [fill=qqffff] (6.,7.) circle (3.5pt);
\draw [fill=ffffff] (7.,0.) circle (2.5pt);
\draw [fill=ffffff] (7.,1.) circle (2.5pt);
\draw [fill=qqffff] (7.,2.) circle (3.5pt);
\draw [fill=ffffff] (7.,3.) circle (2.5pt);
\draw [fill=qqffff] (7.,4.) circle (3.5pt);
\draw [fill=qqffff] (7.,5.) circle (3.5pt);
\draw [fill=ffffff] (7.,6.) circle (2.5pt);
\draw [fill=qqffff] (7.,7.) circle (3.5pt);
\end{scriptsize}
\draw (-1.1,3.5) node[anchor=mid] {$\widehat{Q}$};
\draw (1.5,7.65) node[anchor=mid] {$Q+m(0,1)$};
\draw (1.5,-0.65) node[anchor=mid] {$Q$};
\draw (5.5,7.65) node[anchor=mid] {$Q+m(1,1)$};
\draw (5.5,-0.65) node[anchor=mid] {$Q+m(1,0)$};
\end{tikzpicture}
\caption{An example for $d=2$ and $m=4$, with the subset $S\subseteq Q$ shown in dark blue, and the additional points extending $S$ to $S_c$ shown in light blue. The parts of $\widehat{Q}$ belonging to each vector $\mathbf{v}$ are coloured according to the parity of $\mathbf{v}$.}
\label{fig:checkerboard}
\end{figure}

An analogue of Observation \ref{obs:hypercube_int_property} for the checkered extension of any set is the following:
\begin{lemma}\label{lem:checkered_property}
Let $0\le \ell\le d-1$ be an integer. Given $J=\{j_1, j_2, \ldots, j_{\ell}\}\subseteq [d]$ and integers $k_1, k_2, \ldots, k_{\ell}\in [m]$, and any subset $S\subseteq Q$, the following holds:
$$|S_c\cap H(j_1, k_1)\cap \ldots\cap H(j_{\ell},k_{\ell})|=\frac12\cdot (2m)^{d-\ell}.$$
\end{lemma}
\begin{proof}
Let $F=H(j_1, k_1)\cap \dots\cap H(j_{\ell}, k_{\ell})$. Then $F$ intersects $\widehat{Q}$ in the following set of size $(2m)^{d-\ell}$:
\begin{equation*}
\begin{split}
F\cap \widehat{Q} &=\{\mathbf{y}\in \widehat{Q}: \mathbf{y}_{j_a}=k_a\text{ for each $1\le a\le \ell$}\} \\
&=\{\mathbf{y}=\mathbf{x}+m\mathbf{v}\in \widehat{Q}: (\mathbf{x}_{j_a}=k_a \text{ and }\mathbf{v}_{j_a}=0)\text{ for each $1\le a\le \ell$}\}.
\end{split}
\end{equation*}

There are $m^{d-\ell}$ vectors $\mathbf{x}\in Q$ such that $\mathbf{x}_{j_a}=k_a$ for all $a$, and $2^{d-\ell}$ vectors $\mathbf{v}\in \{0,1\}^d$ such that $\mathbf{v}_{j_a}=0$ for all $a$. Out of the latter vectors $\mathbf{v}$, precisely half of them are even and half odd (as the parity of the coordinate sum can be set using the last non-restricted coordinate).

To get an element of $S_c$, we need to pick even $\mathbf{v}$ in the case $\mathbf{x}\in S$ and odd $\mathbf{v}$ in the case $\mathbf{x}\not\in S$. As we just saw, there are $2^{d-\ell-1}$ good choices for each $\mathbf{x}$, giving $|F\cap S_c|=m^{d-\ell}\cdot 2^{d-\ell-1}=\frac12\cdot (2m)^{d-\ell}$.
\end{proof}

\begin{remark}\label{rem:checkered_property_0}
By the special case $\ell=0$, we have $|S_c|=\frac12\cdot (2m)^d$ for any $S\subseteq Q$.
\end{remark}

\subsection{The vertices and edges}\label{subsec:vertex_labellings}

We now describe in detail the vertices and edges of our variable graph $G$, which is a partial join of $G_L$ and $G_R$. We take the left-hand class $G_L$ to be the empty graph on the $(m^{d+1})$-vertex set $V_L$, which is the disjoint union of the following sets:
$$V_L=\widehat{Q} \sqcup A_1\sqcup A_2\sqcup \ldots\sqcup A_{d+1},$$

\noindent where
\begin{itemize}
\itemsep0em
\item $\widehat{Q}$ is the \textit{extended hypercube} that we just described, labelled as $\widehat{Q}=[2m]^d$.
\item for each $1\le j\le d$, $A_j$ is called the \textit{$j$-th additional set}, and it has cardinality $n_j=m^{j-1}(m-1)$. Altogether, $\sum_{j=1}^d n_j=m^d-1$.
\item $A_{d+1}$ is called the \textit{special additional set}, and it consists of the remaining $n_{d+1}=m^{d+1}-(2^d+1)m^d+1$ vertices. Note that this number is positive because of the assumption $m>2^d$.
\end{itemize}

\noindent We let $A_j=\{a_{j,1}, \dots, a_{j,n_j}\}$ for each $1\le j\le d+1$.

The right-hand class is $G_R=(d+1)K_m$. For each $1\le i\le d+1$, call the vertex set of the $j$-th of these complete graphs $W_j=\{w_{j,1}, w_{j,2}, \ldots, w_{j,m}\}$, so that $V_R=W_1\sqcup \ldots \sqcup W_{d+1}$.\bigskip

\noindent We now describe the edges of the partial join $G$ between $V_L$ and $V_R$, some of which will be always present in $G$ (called \textit{fixed edges}), while others may be either present or not present in $G$ (called \textit{variable edges}). A partial join can be defined by giving the sets $N_L(w)$, that is, the set of neighbours of $w$ in $V_L$, for each $w\in V_R$. The fixed edges will be a subset of $K[V_L, W_1\sqcup \ldots\sqcup W_d]$, while the variable edges will be a subset of $K[V_L, W_{d+1}]$.\bigskip

\noindent For any $j\in [d]$ and $k\in [m]$, let $s(j,k)=(k-1)m^{j-1}$. Let
$$N_L(w_{j,k})=\left(\widehat{Q}\cap H(j,k)\right) ~\sqcup~ \{a_{j,\ell}: 1\le \ell\le s(j,k)\}.$$

\noindent Observe that since $|\widehat{Q}\cap H(j,k)|=(2m)^{d-1}$ for all $j,k$, we have
$$|N_L(w_{j,k})|=(2m)^{d-1}+s(j,k).$$

\noindent Also note that no vertex in $W_1\sqcup \ldots\sqcup W_d$ has any neighbour in $A_{d+1}$.\bigskip

\noindent Now we give the variable edges. For each $1\le k\le m$, we fix an integer $t_k$, which will be chosen later on, with a dependence on $d,m,\eps$. The chosen integers $t_k$ will satisfy
$$2^{d-1}m^d\le t_k\le 2^{d-1}m^d+n_{d+1}=m^{d+1}-(2^{d-1}+1)m^d+1.$$

Now let us choose any sets $S_1, S_2, \ldots, S_m\subseteq Q=[m]^d$, and consider their checkered extensions $(S_1)_c, \ldots, (S_m)_c\subseteq \widehat{Q}=[2m]^d$. The sets $S_k$ can be chosen completely arbitrarily, and they are responsible for creating the variability of the construction, and hence of the Fibonacci number of the partial join $G$.

If $t_k=2^{d-1}m^d+t_k^{*}$ (where $0\le t_k^{*}\le n_{d+1}$), then let
$$N_L(w_{d+1,k})=(S_k)_c \sqcup \{a_{d+1,1}, \ldots, a_{d+1,{t_k^{*}}}\}.$$

In particular, note that no vertex in $W_{d+1}$ has any neighbour in $A_1\sqcup \ldots\sqcup A_d$. We have $|N_L(w_{d+1,k})|=\frac12\cdot (2m)^d+t_k^{*}=t_k$ by Remark \ref{rem:checkered_property_0}.

\subsection{Computation of $i(G)$}\label{subsec:computation}

The goal of this subsection (and the subsequent one) is to prove Proposition \ref{prop:long_interval}, that is, to show that $\Sp(G_L, G_R)$ contains $\cB(J)+c$, where $J$ is a long interval. More specifically, our variable construction will have the property in Proposition \ref{prop:toggles}, from which Proposition \ref{prop:long_interval} directly follows.

We will use the notion of full and partial transversals introduced in Subsection \ref{subsec:intro_transversals} in relation to the graph $G_R$ on vertex set $V_R=W_1\sqcup \ldots \sqcup W_{d+1}$, which satisfy $m_1=\ldots=m_{d+1}=m$. The notation $T=T(k_1, \dots, k_d,k)$ will be a shorthand for the full transversal
$$T=\{w_{1,k_1}, w_{2,k_2}, \ldots, w_{d,k_d}, w_{d+1,k}\},$$ where $k_1,\ldots,k_d,k\in [m]$.

\begin{proposition}\label{prop:toggles}
There is a constant $c\in \nn$, depending on $d,m,\eps$, for which the following holds. There exist (not necessarily distinct) natural numbers $\ell(T)$ for each full transversal $T$, such that
$$i(G)=\left(\sum_T \mathbbm{1}_{A(T)} 2^{\ell(T)}\right)+c,$$
where the sum is taken over all full transversals $T=T(k_1,\ldots, k_d,k)$, and for each full transversal, $A(T)$ is either the event $(k_1, \ldots, k_d)\in S_k$ or $(k_1, \ldots, k_d)\not\in S_k$. Moreover, the set of these natural numbers $\ell(T)$ contains the interval $\left[\left\lceil(1+\eps)2^{d-1}m^d\right\rceil,~ \left\lfloor m^{d+1}-(1+\eps)2^{d-1}m^d\right\rfloor\right]$.
\end{proposition}

Since the properties $(k_1, \dots, k_d)\in S_k$ can be independently chosen to hold or not to hold using the variable edges of our construction, it is indeed sufficient to show Proposition \ref{prop:toggles}.

To compute the number of independent subsets in the partial join $G$, we apply the summation trick in Example \ref{ex:multiplecliques}, bearing in mind that in our case $G_L$ is an empty graph, so we have $i(G_L[X])=2^{|X|}$ for any subset $X\subseteq V_L$:

\begin{equation}\label{eq:ig_as_sum}
i(G)=\sum_T (2^{|V_L\setminus \bigcup_{w\in T} N_L(w)|}: \text{$T$ is a partial transversal}).
\end{equation}

For a partial transversal $T$, let $U_T=\bigcup_{w\in T} N_L(w)$. To calculate $|U_T|$ for a specific $T\ne \emptyset$, we will use the inclusion-exclusion principle:

\begin{equation}\label{eq:inc_exc}
|U_T|=\sum_{\substack{T'\subseteq T\\ T'\ne\emptyset}} (-1)^{|T'|+1} \left|\bigcap_{w\in T'} N_L(w)\right|.
\end{equation}

Letting $T'\subseteq T$ be a nonempty subset and $N=\bigcap\limits_{w\in T'} N_L(w)$, we first present the calculation of $|N|$. Suppose that $T'=\{w_{j_1, k_1}, \ldots, w_{j_{\ell}, k_{\ell}}\}$ for some $1\le j_1 < \ldots<j_{\ell}\le d+1$ and $k_1,\ldots,k_{\ell}\in [m]$. We consider the following two cases.

\begin{description}[style=nextline,leftmargin=1em]
	\item[Case 1: $T'\cap W_{d+1}=\emptyset$.]
    In this case, if $\ell=1$, say $T'=\{w_{j,k}\}$, we clearly have $|N|=|N_L(w_{j,k})|=(2m)^{d-1}+s(j,k)$.

    If $\ell\ge 2$, then since $N_L(w_{j,k})\setminus \widehat{Q}\subseteq A_j$, the sets only intersect in $\widehat{Q}$. The intersection is $N = \widehat{Q}\cap H(j_1, k_1)\cap \ldots\cap H(j_{\ell},k_{\ell})$,
    and by Observation \ref{obs:hypercube_int_property}, we have $|N|=(2m)^{d-\ell}$.

    \item[Case 2: $T'\cap W_{d+1}\ne\emptyset$.]
    Let $T'\cap W_{d+1}=\{w_{d+1, k}\}$, so $j_{\ell}=d+1$ and $k_{\ell}=k$.

    If $\ell=1$, that is, $T'=\{w_{d+1,k}\}$, then clearly $|N|=|N_L(w_{d+1,k})|=t_k$.

    When $\ell\in [2,d]$, the sets only intersect in $\widehat{Q}$. By Lemma \ref{lem:checkered_property}, we have
    $$|N|=|(S_k)_c\cap H(j_1,k_1)\cap \ldots\cap H(j_{\ell-1},k_{\ell-1})|=\frac12\cdot (2m)^{d-\ell+1}.$$

    And when $\ell=d+1$, that is, $T'$ contains one vertex from each $W_j$, then
    $$|N|=|(S_k)_c\cap H(1,k_1)\cap \ldots\cap H(d,k_d)|.$$

    The $d$ hyperplanes $H(1,k_1)$, \ldots, $H(d,k_d)$ meet in a single point $\mathbf{y}=(k_1,\ldots,k_d)\in Q$, and $(S_k)_c\cap Q=S_k$, so whether or not $\mathbf{y}$ is included in $S_k$ decides whether $|N|$ is 1 or 0, that is,
    $$|N|=\mathbbm{1}_{(k_1, \ldots, k_d)\in S_k}.$$
\end{description}

\noindent Now for an arbitrary set $T\ne \emptyset$ in the summation of \eqref{eq:ig_as_sum}, we proceed to calculate $|U_T|$ using \eqref{eq:inc_exc}.

Using our computations of $|N|$ in various cases for the subset $T'$, we obtain that the only case where $|N|$ depended on the choices of $S_1,\ldots,S_m$ was when $|T'|=d+1$. Such a case is only reached in the calculation of $|U_T|$ when $|T|=d+1$, leading to the following observation.

\begin{observation}\label{obs:us_constant_if}
If $0\le |T|\le d$ holds for a set $T$ appearing in the summation of \eqref{eq:ig_as_sum}, then $|U_T|$ is a constant independent of $S_1,\ldots,S_m$.
\end{observation}

So we focus the case $|T|=d+1$, that is, when $T=T(k_1,\ldots,k_d,k)$ is a full transversal. In this case, $|U_T|$ will contain exactly one term which depends on the choices of $S_1,\ldots,S_m$, coming from the case $T'=T$, where we get $|N|=\mathbbm{1}_{(k_1, \ldots, k_d)\in S_k}$. Considering the overall sum, we have
\begin{equation*}
\begin{split}
|U_T| &= \sum_{j=1}^d \left((2m)^{d-1}+s(j,k_j)\right) + \sum_{\substack{|T'|\ge 2 \\ T'\cap W_{d+1}=\emptyset}} (-1)^{|T'|+1} (2m)^{d-|T'|}\\
&+ t_k + \sum_{\substack{|T'|\in [2,d]\\ T'\cap W_{d+1}\ne \emptyset}} (-1)^{|T'|+1}\cdot \frac12\cdot(2m)^{d-|T'|+1} + (-1)^{d+2}\cdot \mathbbm{1}_{(k_1, \ldots, k_d)\in S_k}.
\end{split}
\end{equation*}

Let $\sigma_{k_1,\ldots,k_d}:=\sum_{j=1}^d s(j,k_j)=\sum_{j=1}^d (k_j-1) m^{j-1}$. So $k_1,\ldots,k_d\in [m]$ determine the $d$ rightmost base $m$ digits of $\sigma_{k_1,\ldots,k_d}$, which can hence be any value in $[0, m^d-1]$, with each value in this interval given by exactly one choice of $(k_1, \ldots, k_d)$. We get
\begin{equation*}
\begin{split}
|U_T| &= t_k + \sigma_{k_1, \ldots, k_d} + (-1)^d\cdot \mathbbm{1}_{(k_1, \ldots, k_d)\in S_k} + \sum_{\ell=1}^d (-1)^{\ell+1}\binom{d}{\ell}(2m)^{d-\ell} + \frac12\sum_{\ell=2}^d (-1)^{\ell+1} \binom{d}{\ell-1}(2m)^{d-\ell+1}\\
&= t_k + \sigma_{k_1,\ldots,k_d} + (-1)^d\cdot \mathbbm{1}_{(k_1, \ldots, k_d)\in S_k} + \frac12\left((2m)^d-(2m-1)^d+(-1)^{d+1}\right).
\end{split}
\end{equation*}

So depending on whether $(k_1, \ldots, k_d)\in S_k$ or not, $|U_T|$ is either $\lambda_T-1$ or $\lambda_T$, where \begin{equation}\label{eq:lambdat}
\lambda_T=t_k+\sigma_{k_1,\ldots,k_d}+\frac12\left((2m)^d-(2m-1)^d+1\right),
\end{equation}
and the case giving the larger value depends on the parity of $d$. An easy estimate gives $(2m)^d-(2m-1)^d+1\in \left[d(2m-1)^{d-1}, d(2m)^{d-1}\right]$.

Now in \eqref{eq:ig_as_sum}, we altogether have
$$i(G)=c_1+\sum_T 2^{m^{d+1}-|U_T|}=c_1+\sum_T 2^{m^{d+1}-\lambda_T} + \sum_T 2^{m^{d+1}-\lambda_T}\cdot \mathbbm{1}_{A(T)} = c_2 + \sum_T \mathbbm{1}_{A(T)}\cdot 2^{m^{d+1}-\lambda_T},$$
where $T$ runs through all full transversals, $A(T)$ denotes either the event $(k_1, \ldots, k_d)\in S_k$ or its negation for each $T$, and $c_1,c_2\in \nn$ are constants depending only on $d,m,\eps$. 

So in order to prove Proposition \ref{prop:toggles}, all that remains is to choose the values $t_1, \ldots, t_m$ so that the set $\{\ell(T)\}$ contains the interval
$$J=\left[\left\lceil(1+\eps)2^{d-1}m^d\right\rceil,~ \left\lfloor m^{d+1}-(1+\eps)2^{d-1}m^d\right\rfloor\right],$$ where $\ell(T)=m^{d+1}-\lambda_T$.

\subsection{Fixing the choices of $t_k$}\label{subsec:fixingtkchoices}

Recall from \eqref{eq:lambdat} that for a full transversal $T=T(k_1, \ldots, k_d, k)$, we have  $$\lambda_T=c_3+ t_k + \sigma_{k_1, \ldots, k_d},$$ where $c_3=\frac12((2m)^d-(2m-1)^d+1)\in \nn$ is an overall constant, and $\sigma_{k_1, \ldots, k_d}$ is a different value in $[0,m^d-1]$ for each $(k_1, \ldots, k_d)$.

Also recall that we need to choose the values $t_k$ for each $1\le k\le m$ such that $$2^{d-1}m^d\le t_k\le m^{d+1}-(2^{d-1}+1)m^d+1.$$

If the first $\ell$ values $t_1, t_2, \ldots, t_{\ell}$ form an arithmetic progression of difference $m^d$, that is, $t_k=t_1+(k-1)m^d$ for all $k\le \ell$, this guarantees that the corresponding values $\lambda_T$ for $k\le \ell$ cover a single interval of natural numbers. In particular, let $t_1=2^{d-1}m^d$, meaning that $$t_{m-2^d}=2^{d-1}m^d+(m-2^d-1)m^d=m^{d+1}-(2^{d-1}+1)m^d$$ is the last element that fits within the upper bound.

So we let $\ell=m-2^d$, and we simply set the remaining values $t_k$ ($k>\ell$) to a constant, say $t_k=2^{d-1}m^d$, meaning that the corresponding vertices of $W_{d+1}$ are not actually required for our statements, and it would have sufficed to pick $|W_{d+1}|=m-2^d$. However, we chose $|W_{d+1}|=m$ for the sake of simplicity in presenting the construction.

Altogether, the values $t_k+\sigma_{k_1, \ldots, k_d}$ cover the interval $$J_1=[t_1, t_{m-2^d}+m^d-1]=[2^{d-1}m^d, m^{d+1}-2^{d-1}m^d-1],$$ and we have $c_3\in \left[\frac12d(2m-1)^{d-1}, \frac12d(2m)^{d-1}\right]$, where by Bernoulli's inequality (using $m>2^d$), $(2m-1)^{d-1}\ge 0.99(2m)^{d-1}$ for $d$ large enough. So the values $\lambda_T$ cover the interval
\begin{equation*}
\begin{split}
& \left[2^{d-1}m^d+\frac12d\cdot 2^{d-1}m^{d-1}, m^{d+1}-2^{d-1}m^d+\left\lfloor\frac12\cdot 0.99 d\cdot 2^{d-1}m^{d-1}-1\right\rfloor\right] \\
\supseteq & \left[\left\lceil((1+\eps)2^{d-1}m^d\right\rceil, \left\lfloor m^{d+1}-(1+\eps)2^{d-1}m^d\right\rfloor\right]
\end{split}
\end{equation*}

for $d\ge d_0(\eps)$ large enough, and by symmetry, so do the values $\ell(T)=m^{d+1}-\lambda_T$. This concludes the proof of Propositions \ref{prop:toggles} and \ref{prop:long_interval}.

\section{Construction: combination of multiple parts}\label{sec:constr_combination}

Using the main structure described in the previous section, this section is devoted to completing the proof of the lower bound of Theorem \ref{thm:main}. We combine several copies of our structure to create our final variable graph $G$.

We will take a sequence of values $n=n(D)\to \infty$ indexed by a parameter $D$ that can be equal to any sufficiently large integer. Let $M=2^{D+1}$, and fix the left-hand class to be $G_L=\overline{K_{M^{D+1}}}$, that is, the empty graph on $n_0:=2^{D^2+2D+1}$ vertices.

The construction will be in the spirit of Lemma \ref{lem:join_with_multiple}: we let $G_R$ be the full join of $\ell$ graphs $$G_R=G^{(1)}_R\triangledown G^{(2)}_R \triangledown \dots\triangledown G^{(\ell)}_R,$$ where each graph $G^{(j)}_R$ is isomorphic to $(d_j+1)K_{m_j}$ for some integers $d_j$ and $m_j$, with $m_j>2^{d_j}$. We will use Proposition \ref{prop:long_interval} for each $\Sp(G_L, G^{(j)}_R)$, and then use the lemma to combine the resulting spectra.

The skeleton of our argument is described in the following proposition.

\begin{proposition}\label{prop:combined_skeleton}
Let $n_0\in \zz^{+}$, and take integer intervals $J,Y_0,Y_1,\ldots,Y_k,Z_1,\ldots,Z_k\subseteq [0, n_0]$,  graphs $A_0,A_1,\ldots,A_k$ and integers $n_1, \ldots, n_k\in [n_0]$ such that
\begin{itemize}
\itemsep0em
\item $J\subseteq Y_0\cup Y_1\cup \ldots\cup Y_k\cup Z_1\cup \ldots\cup Z_k$,
\item $Z_j=n_0-Y_j$ for all $1\le j\le k$,
\item there exists $t_j\in [0, n_0-n_j]$ for all $1\le j\le k$ such that $Z_j=Y_j+t_j$,
\item there exist numbers $c_0, c_1, \ldots, c_{k}\in \nn$ such that for each $0\le j\le k$, $$\Sp(\overline{K_{n_j}}, A_j)\supseteq \cB(Y_j)+c_j.$$
\end{itemize}

Then there exists a constant $c\in \nn$ such that $$\Sp(G_L, G_R)\supseteq \cB(J)+c,$$ where:
\begin{itemize}
\itemsep0em
\item $G_L=\overline{K_{n_0}}$,
\item $G_R=A_0\triangledown A_1\triangledown \ldots\triangledown A_k\triangledown B_1\triangledown \ldots\triangledown B_k$, with $B_j$ being an isomorphic copy of $A_j$ for each $1\le j\le k$.
\end{itemize}
\end{proposition}
\begin{proof}
 Firstly we will use Lemma \ref{lem:spectrumscaling} that pads the left-hand class with further isolated vertices, allowing us to obtain a statement about $\Sp(\overline{K_{n_0}}, A_j)$ instead of $\Sp(\overline{K_{n_j}}, A_j)$. Letting $t'_j=n_0-n_j-t_j$, we have
 \begin{itemize}
 \itemsep0em
 \item $\Sp(\overline{K_{n_0}}, A_j)\supseteq 2^{t_j} \Sp(\overline{K_{n_j}}, A_j) + (2^{t'_j}-1)2^{n_j+t_j}\supseteq \cB(Z_j)+2^{t_j}\cdot c_j+(2^{t'_j}-1) 2^{n_j+t_j}=\cB(Z_j)+c^{(1)}_j$,
 \item $\Sp(\overline{K_{n_0}}, B_j)\supseteq \Sp(\overline{K_{n_j}}, B_j) + (2^{n_0-n_j}-1)2^{n_j}\supseteq \cB(Y_j)+c^{(2)}_j$
 \end{itemize}

 for some constants $c^{(1)}_j$ and $c^{(2)}_j$, where the first claim comes from applying part (a) $t_j$ times and then part (b) $t'_j$ times.

 Let us combine these spectra using Lemma \ref{lem:join_with_multiple}:
 \begin{equation*}
 \Sp(G_L, G_R)\supseteq \cB(Y_0) + c_0 + \sum_{j=1}^k \left[\cB(Y_j) + \cB(Z_j) + c_j^{(1)} + c_j^{(2)} \right] - 2k\cdot 2^{n_0}\supseteq \cB(J)+c. \qedhere
 \end{equation*}
\end{proof}

\begin{proof}[Proof of Theorem \ref{thm:main}, lower bound]
Fix a global constant $\delta>0$ to be sufficiently small. Then fix an arbitrary positive integer $D>d_0$, where $d_0=d_0(\delta)$ is sufficiently large, and take $M=2^{D+1}$. We are going to use Proposition \ref{prop:combined_skeleton} for $n_0:=M^{D+1}=2^{(D+1)^2}$, where instead of indices $1\le j\le k$ for $Y_j,Z_j,A_j,n_j$, we will use ordered pairs $(d,q)$ of integers satisfying $d\in [d_0, D-1]$ and $q\in [2d+3]$. So altogether,

$$k=\sum_{d=d_0}^{D-1} (2d+3)=D^2+2D-d_0^2-2d_0.$$

Firstly fix $Y_0:=\left[\left\lceil \left(\frac14+\delta\right) 2^{(D+1)^2}\right\rceil, \left\lfloor \left(\frac34-\delta\right) 2^{(D+1)^2}\right\rfloor \right]$, and correspondingly take $A_0=(D+1)K_M$. By Proposition \ref{prop:long_interval} for $\eps=4\delta$, $\Sp(\overline{K_{n_0}}, A_0)\supseteq \cB(Y_0)+c_0$ for some $c_0$.

Now for each $d\in [d_0, D-1]$ and $q\in [2d+3]$, we will define a positive integer $m_{d,q}\in [2^{d+1}, 2^{d+3}]$ such that by Proposition \ref{prop:long_interval} for $\eps=4\delta$, we will have
$$\Sp(\overline{K_{n_{d,q}}}, A_{d,q})\supseteq \cB(Y_{d,q})+c_{d,q}$$

for a constant $c_{d,q}$, where
\begin{itemize}
\itemsep0em
\item $A_{d,q}:=(d+1) K_{m_{d,q}}$,
\item $n_{d,q}:=m_{d,q}^{d+1}$.
\item $Y_{d,q}:=\left[\left\lceil \frac13\cdot 2^{(d+2)^2-q}\right\rceil, \left\lfloor \frac23\cdot 2^{(d+2)^2-q}\right\rfloor \right]$.
\end{itemize}

The intervals $Y_{d,q}$ for a single value $d\in [d_0, D-1]$ cover the interval $$\left[\left\lceil \frac13\cdot 2^{(d+1)^2}\right\rceil, \left\lfloor \frac23\cdot 2^{(d+2)^2-1}\right\rfloor \right]$$ without gaps, and altogether with $Y_0$, the intervals for all $d$ cover $$\left[\left\lceil \frac13\cdot 2^{(d_0+1)^2}\right\rceil, \left\lfloor\frac23\cdot 2^{(D+1)^2}\right\rfloor \right].$$

Together with the reflections $Z_{d,q}:=2^{(D+1)^2}-Y_{d,q}$, the intervals cover $J:=[a, 2^{(D+1)^2}-a]$, where $a:=\left\lceil \frac13\cdot 2^{(d_0+1)^2}\right\rceil$ is a constant depending on $\delta$.

We apply Proposition \ref{prop:combined_skeleton} for $G_L=\overline{K_{2^{(D+1)^2}}}$ and $G_R=A_0\triangledown A_{d_0,1}\triangledown \ldots\triangledown A_{D-1,2d+3}\triangledown B_{d_0,1}\triangledown \ldots\triangledown B_{D_1,2d+3}$, with $B_{d,q}$ being an isomorphic copy of $A_{d,q}$ for each $d\in [d_0,D-1]$ and $q\in [2d+3]$.

To satisfy the conditions of this proposition, we need to have $Z_{d,q}=Y_{d,q}+t_{d,q}$ for some $0\le t_{d,q} \le 2^{(D+1)^2}-m_{d,q}^{d+1}$, which is equivalent to the following condition:

\begin{enumerate}[label=(\arabic*)]
\item $2^{(d+2)^2-q}\ge m_{d,q}^{d+1}$.
\end{enumerate}

Note that Proposition \ref{prop:long_interval} for $\eps=4\delta$ gives 
$$\Sp(\overline{K_{n_{d,q}}}, A_{d,q})\supseteq  \cB\left(\left[\left\lceil\left(\frac14+\delta\right)2^{d+1}m_{d,q}^d\right\rceil,~ \left\lfloor m_{d,q}^{d+1}-\left(\frac14+\delta\right)2^{d+1}m_{d,q}^d\right\rfloor\right]\right)+c_{d,q}.$$

So all that remains is to define the values $m_{d,q}$ for all $d,q$ so that $Y_{d,q}$ is contained in the interval appearing in this expression. It suffices to make sure that:
\begin{enumerate}[label=(\arabic*)]
\itemsep0em
\setcounter{enumi}{1}
\item $\left(\frac14+\delta\right)2^{d+1}m_{d,q}^d\le \frac13\cdot 2^{(d+2)^2-q}$,
\item $m_{d,q}^{d+1}-\left(\frac14+\delta\right)2^{d+1}m_{d,q}^d\ge \frac23\cdot 2^{(d+2)^2-q}$.
\end{enumerate}

Observe that $\left(\frac14+\delta\right)2^{d+1}m_{d,q}^d\le \left(\frac14+\delta\right)m_{d,q}^{d+1}$, since $m_{d,q}\ge 2^{d+1}$. So it suffices to have:
\begin{enumerate}[label=(\arabic*')]
\itemsep0em
\setcounter{enumi}{1}
\item $\left(\frac14+\delta\right)m_{d,q}^{d+1}\le \frac13\cdot 2^{(d+2)^2-q}$,
\item $\left(\frac34-\delta\right)m_{d,q}^{d+1}\ge \frac23\cdot 2^{(d+2)^2-q}$.
\end{enumerate}

These conditions, together with (1), amount to $\left(\frac89+\delta'\right)\cdot 2^{(d+2)^2-q}\le m_{d,q}^{d+1}\le  2^{(d+2)^2-q}$, where $\delta'$ is a function of $\delta$ with $\delta'\to 0$ as $\delta\to 0$. 

In the cases $q\in \{1, 2d+3\}$, the choices $m_{d,1}=2^{d+3}$ and $m_{d,2d+3}=2^{d+1}$ can be seen to work. In the remaining cases $q\not\in\{1, 2d+3\}$, there exists an integer $m_{d,q}\in [2^{d+1}, 2^{d+3}]$ satisfying these inequalities, since:
\begin{itemize}
\itemsep0em
\item for any two neighbouring values $m$ and $m+1$ in this interval, $\frac{(m+1)^{d+1}}{m^{d+1}}\le \exp\left(\frac{d+1}{m}\right)\le \exp\left(\frac{d+1}{2^{d+1}}\right)\le \frac{1}{\frac89+\delta'}$ for $d\ge 5$ and $\delta'$ small enough;
\item the extremal choices $m=2^{d+1}$ and $m=2^{d+3}$ overshoot the endpoints of the required interval, since:
\begin{itemize}[label=$\circ$]
\itemsep0em
\item $2^{(d+1)^2}<\left(\frac89+\delta'\right)\cdot 2^{d^2+2d+2}$,
\item $2^{(d+1)(d+3)}>2^{d^2+4d+2}$.
\end{itemize}
\end{itemize}

So the requirements of Proposition \ref{prop:combined_skeleton} are met, and we altogether get $\Sp(G_L,G_R)\supseteq \cB(J)+c$ for a constant $c$, meaning that $|\Sp(G_L,G_R)|\ge 2^{2^{(D+1)^2}-2a}$.

Altogether, the number of vertices in our construction is
\begin{equation*}
\begin{split}
n &= |V_L| + |V_R| \\
&= 2^{(D+1)^2} + (D+1)2^{D+1} + 2\left(\sum_{d=d_0}^{D-1} \sum_{q=1}^{2d+3} (d+1)m_{d,q}\right) \\
&= 2^{(D+1)^2} + \Theta(D^2\cdot 2^D),
\end{split}
\end{equation*}

and we obtained that the Fibonacci number of such an $n$-vertex graph $G$ can take on at least $$\scalebox{1.25}{$2^{2^{(D+1)^2}-2a}=2^{n-2^{(1+o(1))\sqrt{\log n}}}$}$$ values, attaining the required lower bound.
\end{proof}

\section{Application to $[k,1]$-avoiders}\label{sec:avoiders}

The bounds on the size of the Fibonacci spectrum described in this article can be used in relation to a previous question raised by the authors \cite{KN25}, which is a generalization of several well-known problems in additive combinatorics and finite geometry, such as finding the minimum size of a blocking set or the maximum size of a cap set.
We recall the problem investigated in this paper \cite{KN25} and formulate it in the binary vector space $\ff_2^n$, although a similar question can be asked over any finite field. A $k$-dimensional affine subspace in $\ff_2^n$ is called a \textit{$k$-flat}.

\begin{question}[The avoidance problem for $ \lb k,t\rb $-flats, Kovács, Nagy \cite{KN25}]
We are given integers $n,m,k,t$ with $1\le k\le n$, $~0\le m\le 2^n$ and $0\le t\le 2^k$. Is it true that for any set $S$ of cardinality $m$ in $\ff_2^n$, there exists a $k$-flat $F_k\subseteq \ff_2^n$ that intersects $S$ in exactly $t$ points?
\end{question}

Such a flat is referred to as a \textit{$[k,t]$-flat induced by $S$}. If $S$ does not induce a $[k,t]$-flat, it is a \textit{$[k,t]$-avoider}. The present authors investigated the \textit{$(n;k,t)$-spectrum} $$\Sp(n;k,t)=\{0\le m\le 2^n : \text{every $S$ with $|S|=m$ induces a $[k,t]$-flat}\},$$ and conjectured that for any fixed pair $(k,t)$, we have $$|\Sp(n;k,t)|=(1-o(1))2^n$$ as $n\to \infty$.

In \cite{KN25}, the conjecture was shown true for $t\in \{0, 2^{k-1}, 2^{k}\}$ when $k\ge 3$, and for all $t$ when $k\in \{1,2\}$. A lower bound was also given on the spectrum size when $t$ is of the form $2^a$ or $3\cdot 2^a$ for an integer $1\le a\le k-2$.

In \cite{codebased}, the first author primarily investigated the smallest undecided case $(k,t)=(3,1)$. As an application of the general probabilistic lower bound from our joint paper \cite{KN25}, which was based on an idea of Guruswami \cite{Guruswami}, he obtained the following result.

\begin{theorem}[\cite{codebased}]\label{thm:old_31spectrum}
There exists $K>0$ such that for sufficiently large $n$, $$|\Sp(n; 3,1)|\le 2^n - K\cdot 2^{n/2}.$$
\end{theorem}

\noindent Furthermore, he gave an explicit construction of $[3,1]$-avoiders based on linear codes, which still gives exponentially many different sizes outside the spectrum.
He also suggested a different construction that uses the independent sets of an arbitrary graph $G$ on $n$ vertices to construct a $[3,1]$-avoider:
\begin{construction}[\cite{codebased}]\label{constr:31avoider}
Let $G=([n],E)$ be a graph. If for every edge $e=uv\in E$ we consider the $(n-2)$-flat $F_e=\{\mathbf{x}\in \ff_2^n: \mathbf{x}_u=\mathbf{x}_v=1\}$, then $S_G=\cup_{e\in E} F_e$ is a $[3,1]$-avoider in $\ff_2^n$ with $|S_G|=2^n-i(G)$.
\end{construction}

Using this construction in tandem with Theorem \ref{thm:main} of the current paper, we now obtain the following improvement to Theorem \ref{thm:old_31spectrum}:
\begin{theorem}\label{thm:new_31spectrum}
For infinitely many $n$, we have 
$$\scalebox{1.25}{$|\Sp(n; 3,1)|\le 2^n - 2^{n-2^{(1+o(1))\sqrt{\log n}}}$}.$$
\end{theorem}

The same bound applies for $[k,1]$-avoiders for all $k\ge 4$, because $[3,1]$-avoiders also avoid $[k,1]$-flats.

\section{Concluding remarks}

In this paper, we showed that $\frac{\log_2\cN i(g(n))}{g(n)}\to 1$ holds in a strong sense for a subsequence $g(n)$ of the natural numbers. 
While our main construction was built on a vertex set of size belonging to just an infinite sequence, we believe that the following also holds: for every $\eps>0$, there is a constant $K_{\eps}$ such that for all $n$, $$\cNi(n)\ge K_{\eps} 2^{n(1-\eps)}.$$ 
It also remained an open problem to determine the growth of 
$2^n/\cN i(n)$.

The main problem has a natural extension for hypergraphs as well. For a hypergraph $H=(V,\mathcal{E})$, a subset $W\subseteq V$ is an  \textit{independent set} if $W$ does not contain any  hyperedge of $H$. The rank of the hypergraph is the maximum cardinality of any of the edges of the hypergraph. 

\begin{question}\label{q:runiform}
Let $r\ge 2$. How many values can $i(H)$ admit for hypergraphs $H$ of rank at most $r$ on $n$ vertices?
\end{question}

As a special case, one might investigate only the $r$-uniform hypergraphs.

In \cite{codebased}, the first author showed that a result on Question \ref{q:runiform} has an analogous application to the $[k,1]$-avoider problem using an analogue statement to Construction \ref{constr:31avoider}. From the general construction of \cite{codebased}, we have that for $n\ge k\ge 3$, if $H$ is any hypergraph on $[n]$ of rank at most $k-1$, there is a $[k,1]$-avoider $S_H\subseteq \ff_2^n$ with $|S_H|=2^n-i(H)$. Therefore, a lower bound for Question \ref{q:runiform} immediately implies an upper bound on $|\Sp(n;k,1)|$ as described in Section \ref{sec:avoiders}.

Our result has some consequences in relation to the theory of independence polynomials as well.
Levit and Mandrescu \cite{Levit2} showed several examples of non-isomorphic graph pairs which have the same independence polynomial. In fact, there are graphs $G$ for which $I_G(x) \equiv I_{\overline{G}}(x)$ but $G\not\cong\overline{G}$ \cite{Hoede}.
Thus it is natural to pose the following question:
\begin{problem}
    What is the order of magnitude of $|\{I_G: G\in \cG_n\}|$?
\end{problem}
Our result can also be interpreted as an exponential lower bound on the number of distinct independence polynomials of $n$-vertex graphs with base arbitrarily close to $2$.

In the same spirit, it would be interesting to determine   $|\{ i(G): G\in \cG_n^*\}|$ for some other more restricted graph families $\cG_n^*$ of $n$-vertex graphs.

A closely related problem would be to determine the set of cardinalities of the total number of independent $k$-sets that can be attained in $n$-vertex graphs. Observe that this question concerns the value of the coefficient $i_k$ of $x^k$ in the independence polynomial $I_G(x)$. This problem has been addressed very recently by Kittipassorn and Popielarz \cite{Kitti} in a more general framework, as they  proved results for the set of the possible number of copies of a fixed connected graph $F$ in a graph on $n$ vertices. Setting $F$ to be a clique would give back the previously mentioned variant of our main problem.\\

\noindent \textbf{Acknowledgement}.
We are grateful to Ferenc Bencs for the fruitful discussion on the topic.

\end{document}